\documentclass{amsart}
\usepackage{amssymb,amscd}
\usepackage[cp1251]{inputenc}
\usepackage[T2A]{fontenc}

\newtheorem{theorem}{Theorem}[section]
\newtheorem{proposition}[theorem]{Proposition}

\newtheorem{lemma}[theorem]{Lemma}

\theoremstyle{definition}

\newtheorem{example}[theorem]{Example}
\newtheorem{construction}[theorem]{Construction}
\theoremstyle{remark}
\newtheorem*{remark}{Remark}

\numberwithin{equation}{section}
\numberwithin{figure}{section}

\def\C{\mathbb C}

\def\R{\mathbb R}
\def\T{\mathbb T}
\def\Z{\mathbb Z}

\def\phi{\varphi}

\newcommand{\mb}[1]{{\textbf {\textit#1}}}

\renewcommand{\ge}{\geqslant}
\renewcommand{\le}{\leqslant}

\newcommand{\zp}{\mathcal Z_P}

\begin{document}

\title[Moment-angle manifolds and Lagrangian embeddings]%
{Intersections of quadrics,\\ moment-angle manifolds, and\\
Hamiltonian-minimal Lagrangian embeddings}

\author{Andrey Mironov}
\address{Sobolev Institute of Mathematics, 4 Acad. Koptyug avenue, 630090 Novosibirsk, Russia,
\quad \emph{and}\newline\indent Laboratory of Geometric Methods in
Mathematical Physics, Moscow State University}
\email{mironov@math.nsc.ru}

\author{Taras Panov}
\address{Department of Mathematics and Mechanics, Moscow
State University, Leninskie Gory, 119991 Moscow, Russia,
\newline\indent Institute for Theoretical and Experimental Physics,
Moscow, Russia,\quad \emph{and}
\newline\indent Institute for Information Transmission Problems,
Russian Academy of Sciences} \email{tpanov@mech.math.msu.su}

\thanks{The first author was supported by grants МД-5134.2012.1 and НШ-544.2012.1 from the
President of Russia. The second author was supported by grants
МД-111.2013.1 and НШ-4995-2012.1 from the President of Russia and
RFBR grant~11-01-00694. Both authors were supported by RFBR
grant~12-01-92104-ЯФ, grants from Dmitri Zimin's `Dynasty'
foundation, and grant no.~2010-220-01-077 of the Government of
Russia.}



\begin{abstract}
We study the topology of Hamiltonian-minimal Lagrangian
submanifolds $N$ in $\C^m$ constructed from intersections of real
quadrics in a work of the first author. This construction is
linked via an embedding criterion to the well-known Delzant
construction of Hamiltonian toric manifolds. We establish the
following topological properties of~$N$: every $N$ embeds as a
submanifold in the corresponding moment-angle manifold $\mathcal
Z$, and every $N$ is the total space of two different fibrations,
one over the torus $T^{m-n}$ with fibre a real moment-angle
manifold~$\mathcal R$, and the other over a quotient of $\mathcal
R$ by a finite group with fibre a torus. These properties are used
to produce new examples of Hamiltonian-minimal Lagrangian
submanifolds with quite complicated topology.
\end{abstract}

\maketitle

\section{Introduction}
In this paper we study the topology of a class of
Hamiltonian-minimal ($H$-minimal) Lagrangian submanifolds in
${\mathbb C}^m$ obtained from intersections of real quadrics.

Let $M$ be a K\"ahler manifold. A Lagrangian submanifold $N\subset
M$ is called \emph{$H$-minimal} if its volume is critical under
Hamiltonian deformations. The simplest example  of an $H$-minimal
Lagrangian submanifold is the Clifford torus~\cite{Oh}
$$
 S^1(r_1)\times\dots \times S^1(r_m)\subset {\mathbb C}^m,
$$
where $S^1(r_k)\subset {\mathbb C}$ is a circle of radius $r_k$.
Other $H$-minimal Lagrangian tori in ${\mathbb C}^2$ were
constructed in~\cite{CU} and~\cite{HR1}, an example of an immersed
$H$-minimal Lagrangian Klein bottle was given in~\cite{HR2}, and
more higher dimensional examples were obtained in~\cite{AC}.
In~\cite{miro04} the first author suggested a universal method for
constructing $H$-minimal Lagrangian immersions
$N\looparrowright\C^m$ from intersections of real
quadrics~$\mathcal R$.
By this method one can construct $H$-minimal Lagrangian immersions
in ${\mathbb C}^m$ of the Klein bottle ${\mathcal K}^m$,
$S^{m-1}\times S^1$, ${\mathcal K}^{m-1}\times S^1$ and other
manifolds.

In this paper we give effective criteria for the map $N\rightarrow
{\mathbb C}^m$ to be an embedding (Theorems~\ref{nembed}
and~\ref{delz}). This is done by exploring the link between
intersections of quadrics, simple polytopes and moment-angle
manifolds, and using methods of \emph{toric
topology}~\cite{bu-pa02},~\cite{pano10}. It turns out that
$N\rightarrow {\mathbb C}^m$ is an $H$-minimal Lagrangian
embedding precisely when the polytope corresponding to the
intersection of quadrics is \emph{Delzant}. Delzant polytopes
classify Hamiltonian toric manifolds~\cite{delz88}, which can be
obtained as symplectic quotients of $\C^m$ by a torus action; our
intersection of quadrics is a level set for the corresponding
moment map (see e.g.~\cite[\S8.2]{bu-pa02}). In
Proposition~\ref{nprop} we show that $N$ is embedded in the
moment-angle manifold $\mathcal Z$ and that $N$ is the total space
of two different fibre bundles. The first bundle is over a torus
with fibre the intersection of quadrics~$\mathcal R$ and the
second is a principal torus bundle over a quotient of $\mathcal R$
by a finite group. The latter quotient is known as a \emph{small
cover} over a simple polytope~\cite{da-ja91}. We also give a
topological classification of manifolds $N$ in the case when
$\mathcal R$ is an intersection of two quadrics (see
Theorem~\ref{2qemb}).

It is natural to ask what closed manifolds can be embedded in
${\mathbb C}^m$ as Lagrangian submanifolds. There are various
topological restrictions for such embeddings. For example, a
manifold $M$ with $H^1(M,{\mathbb R})=0$ cannot be embedded as a
Lagrangian submanifold in ${\mathbb C}^m$~\cite{Gr}. There are no
Lagrangian embeddings of an even-dimensional Klein bottle in
${\mathbb C}^{2m}$~\cite{nemi09} (for the case $m=1$ see
also~\cite{Sh}). Our construction gives a large family of
Lagrangian submanifolds in ${\mathbb C}^m$ of quite complicated
topology. For example, there is a Lagrangian submanifold
in~$\C^{5}$ which is the total space of a fibration over $T^3$
with fibre a surface of genus 5 (see Example~\ref{3quad}).

A modification of the construction of $H$-minimal Lagrangian
submanifolds in $\C^m$ can be used to produce $H$-minimal
Lagrangian submanifolds in $\C P^{m-1}$, see~\cite{miro04}.
Namely, if $N$ is an $H$-minimal Lagrangian cone, then, taking the
intersection of $N$ with the unit sphere and factorising by the
diagonal circle action, we obtain an $H$-minimal Lagrangian
submanifold in ${\mathbb C}P^{m-1}$. Large families of new
explicit examples can be obtained by this procedure. These include
$H$-minimal Lagrangian immersions of tori in ${\mathbb C}P^{2}$
and ${\mathbb C}P^{3}$ described in \cite{Mi1}, \cite{MZ},
\cite{Hu}. For more projective examples see~\cite{CLU}, \cite{CG}
and~\cite{HK}.
As in the case of~$\C^m$, there are topological restrictions for
Lagrangian embeddings in~$\mathbb C P^{m-1}$ (see for example
\cite{Bi},~\cite{Se}).

Hamiltonian stability is an important related property of
Lagrangian submanifolds. According to a result of~\cite{Oh}, the
Clifford torus is Hamiltonian-stable as an $H$-minimal Lagrangian
submanifold in~${\mathbb C}^m$. Other examples of
Hamiltonian-stable submanifolds can be found in~\cite{AO1},
\cite{AO2}. It would be interesting to analyse the Hamiltonian
stability of the $H$-minimal submanifolds~$N$ considered here.

\medskip

We use the following notation throughout the paper:
\begin{itemize}
\item $\Z^m$, $\R^m$, and $\C^m$ are the standard integer lattice of
rank~$m$, the standard real space, and the standard complex space,
respectively;

\item $\T^m=\bigl\{(e^{2\pi i\chi_1},\ldots,e^{2\pi i\chi_m})\in\C^m\bigr\}$, where
$(\chi_1,\ldots,\chi_m)\in\R^m$ is the standard $m$-torus;

\item $[m]=\{1,\ldots,m\}$ is the standard ordered set of $m$
elements;

\item $\Z\langle\mb a_1,\ldots,\mb a_k\rangle$ is the set of integer linear combinations of vectors
$\mb a_1,\ldots,\mb a_k$;

\item $\sigma\langle\mb a_1,\ldots,\mb a_k\rangle$ is the cone (the set of nonnegative $\R$-linear combinations)
generated by vectors $\mb a_1,\ldots,\mb a_k\in\R^m$;

\item $\R^m_\ge=\{(y_1,\ldots,y_m)\in\R^m\colon y_i\ge0
\text{ for all }i\}$ is the standard positive cone (orthant), that
is, the cone generated by the standard basis.
\end{itemize}

The authors are grateful to Stefan Nemirovski for stimulating
discussions of known examples of Lagrangian submanifolds and
Lagrangian non-embeddability results. We thank the referee for
drawing our attention to the work of Y.~Dong~\cite{dong07} and
encouraging us to further explore the relationship between our
construction and symplectic reduction.

\section{Intersections of quadrics}\label{inter}

Assume given a set of $m$ vectors
\[
  \Gamma=\bigl\{\gamma_k=(\gamma_{1,k},\ldots,\gamma_{m-n,k})\in\R^{m-n},\quad
  1\le k\le m\bigr\}
\]
and a vector $\mb c=(c_1,\ldots,c_{m-n})\in\R^{m-n}$. We define
the following intersections of $m-n$ quadrics in $\R^{m}$ and
$\C^m$, respectively:
\begin{align}
  \mathcal R_\Gamma&=\Bigl\{\mb u=(u_1,\ldots,u_m)\in\R^m\colon
  \sum_{k=1}^m\gamma_{jk}u_k^2=c_j,\quad\text{for }
  1\le j\le m-n\Bigr\},\label{rquad}\\
  \mathcal Z_\Gamma&=\Bigl\{\mb z=(z_1,\ldots,z_m)\in\C^m\colon
  \sum_{k=1}^m\gamma_{jk}|z_k|^2=c_j,\quad\text{for }
  1\le j\le m-n\Bigr\}\label{zquad}.
\end{align}

We consider these intersections of quadrics up to \emph{linear
equivalence}, which corresponds to applying a nondegenerate linear
transformation of $\R^{m-n}$ to $\Gamma$ and~$\mb c$. Obviously,
such a linear equivalence does not change the sets $\mathcal
R_\Gamma$ and~$\mathcal Z_\Gamma$.


A version of the following proposition appeared in~\cite{lope89},
and its proof is a modification of the argument
in~\cite[Lemma~0.3]{bo-me06}.
\begin{proposition}\label{rzsmooth}
The intersections of quadrics~\eqref{rquad} and~\eqref{zquad} are
nonempty and nondegenerate if and only if the following two
conditions are satisfied:
\begin{itemize}
\item[(a)] $\mb c\in
\sigma\langle\gamma_1,\ldots,\gamma_m\rangle$;\\[-0.7\baselineskip]

\item[(b)] if $\mb
c\in\sigma\langle\gamma_{i_1},\ldots\gamma_{i_k}\rangle$, then
$k\ge m-n$.
\end{itemize}
Under these conditions, $\mathcal R_\Gamma$ and $\mathcal
Z_\Gamma$ are smooth submanifolds in $\R^m$ and $\C^m$ of
dimension $n$ and $m+n$, respectively, and the vectors
$\gamma_1,\ldots,\gamma_m$ span~$\R^{m-n}$.
\end{proposition}
\begin{proof}
We give a proof for the case of~$\mathcal R_\Gamma$; the case of
$\mathcal Z_\Gamma$ is treated similarly. First, assume that (a)
and (b) are satisfied. Then (a) implies that $\mathcal
R_\Gamma\ne\varnothing$. Let $\mb u\in\mathcal R_\Gamma$. Then the
rank of the matrix of gradients of~\eqref{rquad} at $\mb u$ equals
to
\[
  \mathop{\mathrm{rk}}\{\gamma_k\colon u_k\ne0\}.
\]
We have $\mb c\in\sigma\langle\gamma_k\colon u_k\ne0\rangle$. By
the Carath\'eodory theorem, $\mb c$ is in a cone generated by some
$m-n$ of those vectors, that is, $\mb
c\in\sigma\langle\gamma_{k_1},\ldots,\gamma_{k_{m-n}} \rangle$,
where $u_{k_i}\ne0$ for $i=1,\ldots,m-n$. Moreover, the vectors
$\gamma_{k_1},\ldots,\gamma_{k_{m-n}}$ are linearly independent
(otherwise, again by the Carath\'eodory theorem, we obtain a
contradiction with~(b)). This implies that the gradients of $m-n$
quadrics in~\eqref{rquad} are linearly independent at~$\mb u$, and
therefore $\mathcal R_\Gamma$ is smooth $n$-dimensional.

To prove the other implication we observe that if (b) fails, that
is, $\mb c$ is in the cone generated by some $m-n-1$ vectors of
$\gamma_1,\ldots,\gamma_m$, then there is a point $\mb
u\in\mathcal R_\Gamma$ with at least $n+1$ zero coordinates. The
gradients of quadrics in~\eqref{rquad} cannot be linearly
independent at such~$\mb u$.
\end{proof}

The torus $\T^m$ acts on $\mathcal Z_\Gamma$ coordinatewise.
Similarly, the `real torus' $(\Z/2)^m\subset\T^m$ (corresponding
to $(\chi_1,\ldots,\chi_m)\in\frac12\Z^m$) acts on~$\mathcal
R_\Gamma$.

From now on we assume that the conditions of
Proposition~\ref{rzsmooth} are satisfied. Moreover, we assume that
\begin{itemize}
\item[(c)]
the vectors $\gamma_1,\ldots,\gamma_m$ generate a lattice $L$
in~$\R^{m-n}$.
\end{itemize}
Obviously, conditions (a)--(c) are invariant under linear
equivalence. The lattice $L$ has full rank by
Proposition~\ref{rzsmooth}, that is, $L\cong\Z^{m-n}$. By changing
$\Gamma$ in its linear equivalence class if necessary, we may
assume that $L$ is the standard lattice $\Z^{m-n}\subset\R^{m-n}$.
Let
\[
  L^*=\{\lambda^*\in{\mathbb R}^{m-n}\colon
  \langle\lambda^*,\lambda\rangle\in{\mathbb Z}\text{ for all }\lambda\in L\}
\]
be the dual lattice.

The vectors $\gamma_i$ define an $(m-n)$-dimensional torus
subgroup in $\T^m$, namely,
\[
  T_\Gamma=\bigl\{\bigr(e^{2\pi i\langle\gamma_1,\varphi\rangle},
  \ldots,e^{2\pi i\langle\gamma_m,\varphi\rangle}\bigl)\in\T^m\bigr\},
  \quad\text{where }\varphi\in\R^{m-n};
\]
the lattice of characters of $T_\Gamma$ is~$L$. We shall represent
elements of $T_\Gamma$ by $\varphi\in\R^{m-n}$, noting that
$T_\Gamma$ is identified with the quotient $\R^{m-n}/\displaystyle
L^*$. We define
\[
  D_\Gamma=\frac12 L^*/L^*\cong(\Z/2)^{m-n}.
\]
Note that $D_\Gamma$ embeds canonically as a subgroup in
$T_\Gamma$.

For any $\mb z=(z_1,\ldots,z_m)\in\mathcal Z_\Gamma$, we have the
sublattice
\[
  L_{\mb z}=\Z\langle\gamma_k\colon z_k\ne0\rangle\subset
  L=\Z\langle\gamma_1,\ldots,\gamma_m\rangle.
\]
The following lemma is proved by the same argument as
Proposition~\ref{rzsmooth}.

\begin{lemma}\label{fullrank}
For any $\mb z\in\mathcal Z_\Gamma$, the sublattice $L_{\mb z}$
has full rank~$m-n$, and the same holds for any $\mb u\in\mathcal
R_\Gamma$ and the sublattice $L_{\mb u}$.
\end{lemma}

Recall that an action of a group $G$ on a space $X$ is
\emph{almost free} if all isotropy subgroups are finite.

\begin{proposition}\label{afree}
The group $T_\Gamma$ acts on $\mathcal Z_\Gamma$ almost freely.
Moreover, the isotropy subgroup of $\mb z\in\mathcal Z_\Gamma$ is
given by $\displaystyle L^*_{\mb z}/L^*$.
\end{proposition}
\begin{proof}
An element $(e^{2\pi
i\langle\gamma_1,\varphi\rangle},\ldots,e^{2\pi
i\langle\gamma_m,\varphi\rangle})\in T_\Gamma$ fixes given $\mb
z\in\mathcal Z_\Gamma$ if and only if $e^{2\pi
i\langle\gamma_k,\varphi\rangle}=1$ whenever $z_k\ne0$. This
condition is equivalent to $\langle\gamma_k,\varphi\rangle\in\Z$
whenever $z_k\ne0$, that is, $\varphi\in\displaystyle L^*_{\mb
z}$. Since $\varphi\in L^*$ maps to $1\in T_\Gamma$, the isotropy
subgroup of $\mb z$ is identified with $\displaystyle L^*_{\mb
z}/L^*$. This group is finite by Lemma~\ref{fullrank}.
\end{proof}

\section{Lagrangian immersions}\label{lagri}
Here we briefly review the construction of~\cite{miro04}, which
gives an $H$-minimal Lagrangian immersion in~$\C^m$ for every
intersection of quadrics $\mathcal R_\Gamma$ satisfying conditions
(a)--(c) of the previous section.

Let $M$ be a symplectic manifold of dimension $2n$ with symplectic
form~$\omega$. An immersion $i\colon N\looparrowright M$ of an
$n$-dimensional manifold $N$ is called \emph{Lagrangian} if
$i^*(\omega)=0$. If $i$ is an embedding, then its image is called
a \emph{Lagrangian submanifold} of~$M$. A vector field $\xi$ on
$M$ is \emph{Hamiltonian} if the 1-form $\omega(\:\cdot\:,\xi)$ is
exact.

Assume that a compatible Riemannian metric is chosen on~$M$. A
Lagrangian immersion $i\colon N\looparrowright M$ is called
\emph{Hamiltonian minimal} (\emph{$H$-minimal}) if the variations
of the volume of $i(N)$ along all Hamiltonian vector fields with
compact support are zero, that is,
\[
  \frac d{dt}\mathop{\mathrm{vol}}(i_t(N))\big|_{t=0}=0,
\]
where $i_0(N)=i(N)$, $i_t(N)$ is a deformation of $i(N)$ along a
Hamiltonian vector field, and $\mathop{\mathrm{vol}}(i_t(N))$ is
the volume of the deformed part of $i_t(N)$. An immersion is
\emph{minimal} if the variations of the volume of $i(N)$ along
\emph{all} vector fields are zero.

We endow $\C^m$ with the standard Hermitian metric $\sum_{k=1}^m
d\overline{z}_k\otimes dz_k$. Its real part is the standard
Riemannian metric, and the imaginary part is the standard
symplectic form $\frac i2\sum_{k=1}^m dz_k\wedge d\overline{z}_k$
on~$\C^m$.

Now, returning to the notation of the previous section, we
consider the map
\begin{align*}
  j\colon\mathcal R_\Gamma\times T_\Gamma &\longrightarrow \C^m,\\
  (\mb u,\varphi) &\mapsto \mb u\cdot\varphi=(u_1e^{2\pi
i\langle\gamma_1,\varphi\rangle},\ldots,u_me^{2\pi
i\langle\gamma_m,\varphi\rangle}).
\end{align*}
Note that $j(\mathcal R_\Gamma\times T_\Gamma)\subset\mathcal
Z_\Gamma$. We let $D_\Gamma$ act on $\mathcal R_\Gamma\times
T_\Gamma$ diagonally; this action is free, since it is free on the
second factor. The quotient,
\[
  N_\Gamma=\mathcal R_\Gamma\times_{D_\Gamma} T_\Gamma,
\]
is an $m$-dimensional manifold.

\begin{lemma}\label{immer}\

{\rm (1)} The map $j\colon\mathcal R_\Gamma\times T_\Gamma
\to\C^m$ induces an immersion $i_\Gamma\colon
N_\Gamma\looparrowright\C^m$.

{\rm (2)} The immersion $i_\Gamma$ is an embedding if and only if
$L_{\mb u}=L$ for every $\mb u\in\mathcal R_\Gamma$.
\end{lemma}
\begin{proof}
Take $\mb u\in\mathcal R_\Gamma$, $\varphi\in T_\Gamma$ and $g\in
D_\Gamma$. We have $\mb u\cdot g\in\mathcal R_\Gamma$, and $j(\mb
u\cdot g,g\varphi)=\mb u\cdot g^2\varphi=\mb u\cdot\varphi=j(\mb
u,\varphi)$. Hence the map $j$ is constant on $D_\Gamma$-orbits,
and therefore it induces a map of the quotient $N_\Gamma=(\mathcal
R_\Gamma\times T_\Gamma)/D_\Gamma$, which we denote by~$i_\Gamma$.

Assume that $j(\mb u,\varphi)=j(\mb u',\varphi')$. Then $L_{\mb
u}=L_{\mb u'}$ and
\begin{equation}\label{uu'}
  u_ke^{2\pi i\langle\gamma_k,\varphi\rangle}=u'_ke^{2\pi
  i\langle\gamma_k,\varphi'\rangle}\quad
  \text{for }k=1,\ldots,m.
\end{equation}
Since both $u_k$ and $u'_k$ are real, this implies that $e^{2\pi
i\langle\gamma_k,\varphi-\varphi'\rangle}=\pm1$ whenever
$u_k\ne0$, or, equivalently,
$\varphi-\varphi'\in\frac12\displaystyle L^*_{\mb u}/L^*$. In
other words,~\eqref{uu'} implies that $\mb u'=\mb u\cdot g$ and
$\varphi'=g\varphi$ for some $g\in\frac12{\displaystyle L^*_{\mb
u}/L^*}$. The latter is a finite group by Lemma~\ref{afree}; hence
the preimage of any point of $\C^m$ under $j$ consists of a finite
number of points. If $L_{\mb u}=L$, then $\frac12{\displaystyle
L^*_{\mb u}/L^*}=\frac12\displaystyle L^*/L^*=D_\Gamma$; hence
$(\mb u,\varphi)$ and $(\mb u',\varphi')$ represent the same point
in~$N$. Statement~(2) follows; to prove~(1), it remains to observe
that we have $L_{\mb u}=L$ for generic $\mb u$ (with all
coordinates nonzero).
\end{proof}

\begin{theorem}[{\cite[Theorem~1]{miro04}}]\label{hmin}
The immersion $i_\Gamma\colon N_\Gamma\looparrowright\C^m$ is
$H$-minimal Lagrangian. Moreover, if $\sum_{k=1}^m\gamma_k=0$,
then $i_\Gamma$ is a minimal Lagrangian immersion.
\end{theorem}

Minimal immersions corresponding to one quadric were considered
in~\cite{J}.

\section{Lagrangian embeddings and moment-angle manifolds}
We start by summarising the observations of the previous sections
in the following criterion for $N_\Gamma$ to be embeddable as a
Lagrangian submanifold in~$\C^m$.

\begin{theorem}\label{nembed}
The following conditions are equivalent:
\begin{itemize}
\item[(1)] $i_\Gamma\colon N_\Gamma\to\C^m$ is an embedding of an
$H$-minimal Lagrangian submanifold;

\item[(2)] $L_{\mb u}=L$ for every $\mb u\in\mathcal R_\Gamma$;

\item[(3)] $T_\Gamma$ acts on $\mathcal Z_\Gamma$ freely.
\end{itemize}
\end{theorem}
\begin{proof}
Equivalence (1)$\,\Leftrightarrow\,$(2) follows from
Lemma~\ref{immer} and Theorem~\ref{hmin}. Equivalence
(2)$\,\Leftrightarrow\,$(3) follows from Proposition~\ref{afree}.
\end{proof}

This result allows us to construct explicitly new families of
$H$-minimal Lagrangian submanifolds, once we have an effective
method for producing nondegenerate intersections of quadrics
$\mathcal R_\Gamma$ satisfying conditions~(2) or~(3) of
Theorem~\ref{nembed}. Toric topology provides such a method, which
we describe below following~\cite{bo-me06} and~\cite{pano10}.

The quotient of $\mathcal R_\Gamma$ by the action of $(\Z/2)^m$
(or the quotient of $\mathcal Z_\Gamma$ by the action of $\T^m$)
is identified with the set $P$ of nonnegative solutions of the
following system of $m-n$ linear equations:
\begin{equation}\label{linsys}
  \sum_{k=1}^m\gamma_ky_k=\mb c.
\end{equation}
This set can be described as a convex polyhedron obtained by
intersecting $m$ halfspaces in~$\R^n$:
\begin{equation}\label{ptope}
  P=\bigl\{\mb x\in\R^n\colon\langle\mb a_i,\mb
  x\rangle+b_i\ge0\quad\text{for }
  i=1,\ldots,m\bigr\},
\end{equation}
where $(b_1,\ldots,b_m)$ is any solution of~\eqref{linsys} and the
vectors $\mb a_1,\ldots,\mb a_m\in\R^n$ form the transpose of the
matrix formed by a basis of solutions of the homogeneous system
$\sum_{k=1}^m\gamma_ky_k=\mathbf 0$. We note that $P$ may be
unbounded; in fact $P$, is bounded if and only if $\mathcal
R_\Gamma$ is bounded (compact). Bounded polyhedra are known as
\emph{polytopes}.

\begin{proposition}\label{propcf}
The intersection of quadrics $\mathcal R_\Gamma$ is bounded if and
only it is linear equivalent to an intersection of the following
form:
\begin{equation}\label{canform}
\mathcal R_\Gamma= \left\{\begin{array}{ll}
  \mb u\in\R^m\colon&\gamma_{11}u_1^2+\cdots+\gamma_{1m}u_m^2=c_1,\\[1mm]
  &\gamma_{j1}u_1^2+\cdots+\gamma_{jm}u_m^2=0,\quad\text{for }
  2\le j\le m-n.
  \end{array}\right\},
\end{equation}
where $c_1>0$ and $\gamma_{1k}>0$ for all $k$.
\end{proposition}
\begin{proof}
The quotient of $\mathcal R_\Gamma$ by $(\Z/2)^m$ is the
intersection of the $n$-dimensional affine plane $L$ given
by~\eqref{linsys} with~$\R^m_\ge$. It is bounded if and only if
$L_0\cap\R_\ge^m=\{\mathbf0\}$, where $L_0$ is the $n$-plane
through~$\bf0$ parallel to~$L$. Choose a hyperplane $H_0$ through
$\bf0$ separating the convex sets $L_0$ and~$\R_\ge^m$, that is,
$L_0\subset H_0$ and $H_0\cap\R_\ge^m=\{\mathbf0\}$. Let $H$ be
the affine hyperplane parallel to $H_0$ and containing~$L$. Since
$L\subset H$, we may take the equation defining $H$ as the first
equation in~\eqref{linsys}. The conditions on $H_0$ imply that
$H\cap\R_\ge^m$ is nonempty and bounded, that is, $c_1>0$ and
$\gamma_{1k}>0$ for all~$k$. Now, subtracting the first equation
from the other equations in~\eqref{linsys} with appropriate
coefficients, we achieve that $c_j=0$ for $2\le j\le m-n$.
\end{proof}

We refer to~\eqref{ptope} as a \emph{presentation} of the
polyhedron~$P$ by inequalities. These inequalities contain
slightly more information than the geometric set~$P$, for the
following reason. It may happen that some of the inequalities
$\langle\mb a_i,\mb x\rangle+b_i\ge0$ can be removed from the
presentation without changing the set~$P$; we refer to such
inequalities as \emph{redundant}. A presentation without redundant
inequalities is \emph{irredundant}. Every polyhedron has a unique
irredundant presentation; however, in order to cover all
nondegenerate intersections of quadrics $\mathcal R_\Gamma$, we
need to consider redundant presentations as well.

A presentation~\eqref{ptope} is said to be \emph{generic} if $P$
is $n$-dimensional, has at least one vertex, and the hyperplanes
defined by the equations $\langle\mb a_i,\mb x\rangle+b_i=0$ are
in general position at every point of~$P$. If $P$ is a polytope,
then the existence of a generic presentation implies that $P$ is
\emph{simple}, that is, exactly $n$ facets meet at every vertex
of~$P$. A generic presentation may contain redundant inequalities,
but, for every such inequality, the intersection of the
corresponding hyperplane with $P$ is empty (that is, the
inequality is strict for every $\mb x\in P$).

\begin{theorem}[{see \cite[Lemma~0.12]{bo-me06} or~\cite[Theorem~4.3]{pa-us10}}]
The intersections of quadrics \eqref{rquad} and~\eqref{zquad} are
nondegenerate and nonempty if and only if
presentation~\eqref{ptope} is generic.
\end{theorem}

Conversely, given a generic presentation~\eqref{ptope} of a
polyhedron~$P$, we can reconstruct the intersections of quadrics
$\mathcal R_\Gamma$ and $\mathcal Z_\Gamma$ as follows.

\begin{construction}[moment-angle manifold~{\cite[\S6.1]{bu-pa02}}]\label{dist}
Consider the affine map
\[
  i_P\colon \R^n\to\R^m,\quad \mb x\mapsto
  \bigl(\langle\mb a_1,\mb x\rangle+b_1,\ldots,
  \langle\mb a_m,\mb x\rangle+b_m\bigr).
\]
It is a monomorphism onto a certain $n$-dimensional plane in
$\R^m$ (because $P$ has a vertex), and $i_P(P)$ is the
intersection of this plane with~$\R_\ge^m$.

We define the space $\mathcal Z_P$ from the commutative diagram
\begin{equation}\label{cdiz}
\begin{CD}
  \mathcal Z_P @>i_Z>>\C^m\\
  @VVV\hspace{-0.2em} @VV\mu V @.\\
  P @>i_P>> \R^m_\ge
\end{CD}
\end{equation}
where $\mu(z_1,\ldots,z_m)=(|z_1|^2,\ldots,|z_m|^2)$. This map may
be thought of as the quotient map for the coordinatewise action of
the torus $\T^m$ on~$\C^m$. Therefore, $\T^m$ acts on $\zp$ with
quotient~$P$, and $i_Z$ is a $\T^m$-equivariant embedding.

If~\eqref{ptope} is generic, then $\mathcal Z_P$ is a smooth
manifold of dimension $m+n$, known as the \emph{(polyhedral)
moment-angle manifold} corresponding to~$P$.

Now we can write the $n$-dimensional plane $i_P(\R^n)$ by $m-n$
linear equations in~$\R^m$ as in~\eqref{linsys}. Replacing each
$y_k$ by $|z_k|^2$, we obtain a presentation of the moment-angle
manifold $\zp$ as an intersection of quadrics~\eqref{zquad}.

If we replace $\C^m$ by $\R^m$ in~\eqref{cdiz}, then we obtain the
\emph{real moment-angle manifold}~$\mathcal R_P$. It can be
written as an intersection of quadrics~\eqref{rquad}.
\end{construction}

It is clear from the above constructions that
$\gamma_1,\ldots,\gamma_m$ generate a lattice $L$ in $\R^{m-n}$ if
and only if the vectors $\mb a_1,\ldots,\mb a_m$ in~\eqref{ptope}
generate a lattice $\Lambda$ in~$\R^n$. The corresponding
polyhedra~$P$ are known as \emph{rational}. If $P$ is rational,
then there is a map of lattices
\begin{equation}\label{aplatt}
  A_P\colon\Lambda^*\to\Z^m,\quad
  \mb x\mapsto\bigl(\langle\mb a_1,\mb x\rangle,\ldots,
  \langle\mb a_m,\mb x\rangle\bigr).
\end{equation}
Its conjugate gives rise to a map of tori
$\R^m/\Z^m\to\R^n/\Lambda$, whose kernel we denote by~$T_P$. It
becomes $T_\Gamma$ under the identification of $\zp$ with
$\mathcal Z_\Gamma$. The group $D_P\cong(\Z/2)^{m-n}$ is also
defined. Using $\mathcal R_P$, $T_P$ and $D_P$, we can define the
$m$-dimensional manifold $N_P$ as described in
Section~\ref{lagri}.

The manifolds $\mathcal R_P,\mathcal Z_P,N_P$ therefore represent
the same geometric objects as $\mathcal R_\Gamma,\mathcal
Z_\Gamma,N_\Gamma$, although a different initial data is used in
their definition. From now on, we shall use subscripts $P$ or
$\Gamma$ in the notation of these manifolds only when it is
necessary to emphasise the polyhedral or quadrics origin. In other
cases we shall use the simplified notation $\mathcal R,\mathcal
Z,N$.

We can finally restate the embedding conditions of
Theorem~\ref{nembed} in terms of~$P$. A polyhedron~\eqref{ptope}
is called \emph{Delzant} if it is rational and for every $\mb x\in
P$ the vectors $\mb a_{j_1},\ldots,\mb a_{j_k}$ constitute a part
of a basis of $\Lambda=\Z\langle\mb a_1,\ldots,\mb a_m\rangle$
whenever $\langle\mb a_{j_l},\mb x\rangle+b_{j_l}=0$ for $1\le
l\le k$. (It is enough to verify this condition for vertices $\mb
x\in P$ only, in which case the corresponding $n$-tuple $\mb
a_{j_1},\ldots,\mb a_{j_n}$ must constitute a basis of~$\Lambda$.)
The name refers to a construction~\cite{delz88} of Hamiltonian
toric manifolds.

\begin{theorem}\label{delz}
The map $N_P=\mathcal R_P\times_{D_P}T_P\to\C^m$ is an embedding
if and only if $P$ is a Delzant polyhedron.
\end{theorem}
\begin{proof}
Take $\mb u\in\mathcal R_P$. It projects onto $\mb x\in P$, where
$\langle\mb a_i,\mb x\rangle+b_i=u_i^2$ for $1\le i\le m$. Assume
that exactly $k$ of these numbers vanish. Let
$\iota\colon\Z^{m-k}\rightarrow\Z^m$ be the inclusion of the
coordinate sublattice corresponding to the nonzero~$u_i$, and let
$\kappa\colon\Z^m\rightarrow\Z^k$ be the quotient projection.
Consider the diagram
\[
\begin{CD}
  @.@.\begin{array}{c}0\\ \raisebox{2pt}{$\downarrow$}\\
  \Lambda^*\end{array}@.@.\\
  @.@.@VVA_PV@.@.\\
  0@>>>\Z^{m-k}@>\iota>>\Z^m@>\kappa>>\Z^k @>>>0\\
  @.@.@VV \Gamma V@.@.\\
  @.@.\begin{array}{c}L\\ \downarrow\\ 0\end{array}@.@.\\
\end{CD}
\]
in which the vertical and horizontal sequences are exact, the map
$A_P$ is given by~\eqref{aplatt}, and $\Gamma$ takes the $k$th
basis vector of $\Z^m$ to $\gamma_k$. Then the Delzant condition
is equivalent to that the composition $\kappa\cdot A_P$ is
surjective, while the second condition of Theorem~\ref{nembed} is
that $\Gamma\cdot\iota$ is surjective. A simple diagram chase (see
also~\cite[Theorem~I.2]{pano10}) shows that these two conditions
are equivalent.
\end{proof}

\begin{remark}
When the polytope $P$ is Delzant, the moment-angle manifold $\zp$
given by intersection of quadrics~\eqref{zquad} coincides with the
level set $\mu_P^{-1}(\mb c)$ for the moment map
$\mu_P\colon\C^m\to\R^{m-n}$ used in the construction of the
Hamiltionian toric manifold $M_P=\mu^{-1}_P(\mb c)/T_P$ via
symplectic reduction (see, e.g.~\cite[\S8.2]{bu-pa02}). By using
this observation, the $H$-minimality result of Theorem~\ref{hmin}
in the case of Delzant polytope~$P$ can also be deduced from a
result of Dong~\cite[Corollary~2.7]{dong07}.
\end{remark}

Toric topology provides large families of explicitly constructed
Delzant polytopes. Basic examples include simplices and cubes in
all dimensions. It is easy to see that the Delzant condition is
preserved under several operations on polytopes, such as taking
products or cutting vertices or faces by well-chosen hyperplanes.
This is sufficient to show that many important families of
polytopes, such as \emph{associahedra} (Stasheff polytopes),
\emph{permutahedra}, and general \emph{nestohedra}, admit Delzant
realisations (see~\cite{post09} and~\cite{buch08}).

\section{Topology of Lagrangian submanifolds~$N$}
In the previous section we gave a construction of an $H$-minimal
Lagrangian submanifold in $\C^m$ from any Delzant polytope~$P$.
The moment-angle manifolds $\zp$ and $\mathcal R_P$, appearing as
intermediate objects in this construction, are known to be very
complicated topologically, see~\cite{bu-pa02} and~\cite{pano10}.
Thus, there is no hope for a reasonable topological classification
of Lagrangian submanifolds obtained by our construction. However,
in some cases the topology of $N$ can be described quite
explicitly, providing new examples of $H$-minimal Lagrangian
submanifolds.

We start by reviewing three simple properties linking the
topological structure of the Lagrangian submanifold $N$ to that of
the manifolds $\mathcal Z$ and~$\mathcal R$.

\begin{proposition}\label{nprop}\
\begin{itemize}
\item[(1)] The immersion of $N$ in $\C^m$ factors as
$N\looparrowright \mathcal Z\hookrightarrow\C^m$;
\item[(2)] $N$ is the total space of a bundle over the torus
$T^{m-n}$ with fibre $\mathcal R$;
\item[(3)] if $N\to\C^m$ is an embedding, then $N$ is the total space of a principal
$T^{m-n}$-bundle over the $n$-dimensional manifold $\mathcal
R/D_P$.
\end{itemize}
\end{proposition}
\begin{proof}
Statement (1) is clear. Since $D_P$ acts freely on $T_P$, the
projection $N=\mathcal R\times_{D_P}T_P\to T_P/D_P$ onto the
second factor is a fibre bundle with fibre~$\mathcal R$. Then~(2)
follows from the fact that $T_P/D_P\cong T^{m-n}$.

If $N\to\C^m$ is an embedding, then $T_P$ acts freely on~$\mathcal
Z$ by Theorem~\ref{nembed} and the action of $D_P$ on $\mathcal R$
is also free. Therefore, the projection $N=\mathcal
R\times_{D_P}T_P\to \mathcal R/D_P$ onto the first factor is a
principal $T_P$-bundle, which proves~(3).
\end{proof}

\begin{remark}
The quotient $\mathcal R_P/D_P$ is the \emph{real toric variety},
or the \emph{small cover}, corresponding to the Delzant
polytope~$P$. These manifolds have been studied along with
nonsingular toric varieties, see~\cite{da-ja91}
and~\cite{bu-pa02}.
\end{remark}

\begin{example}[one quadric]\label{1quad}
Suppose that $m-n=1$, that is, $\mathcal R$ is given by a single
equation
\begin{equation}\label{1q}
  \gamma_1u_1^2+\cdots+\gamma_mu_m^2=c
\end{equation}
in $\R^m$, where $\gamma_k\in\R$. If $\mathcal R$ is compact, then
$\mathcal R\cong S^{m-1}$ is a sphere, and the corresponding
polytope $P$ is an $n$-simplex~$\Delta^n$. In this case, $N\cong
S^{m-1}\times_{\Z/2}S^1$, where the generator of $\Z/2$ acts by
the standard free involution on $S^1$ and by a certain involution
$\tau$ on~$S^{m-1}$. The topological type of $N$ depends
on~$\tau$. Namely,
\[
  N\cong\begin{cases}S^{m-1}\times S^1&\text{if $\tau$ preserves the orientation of }S^{m-1},\\
  \mathcal K^{m}&\text{if $\tau$ reverses the orientation of }S^{m-1},\end{cases}
\]
where $\mathcal K^m$ is the \emph{$m$-dimensional Klein bottle}.

\begin{proposition}\label{1qemb}
In the case $m-n=1$ (one quadric) we obtain an $H$-minimal
Lagrangian embedding of $N\cong S^{m-1}\times_{\Z/2}S^1$ in $\C^m$
if and only if $\gamma_1=\cdots=\gamma_m$ in~\eqref{1q}. In this
case, the topological type of $N=N(m)$ depends only on the parity
of~$m$ and is given by
\begin{align*}\label{m=1}
  N(m)&\cong S^{m-1}\times S^1&&\text{if $m$ is even},\\
  N(m)&\cong\mathcal K^{m}&&\text{if $m$ is odd}.
\end{align*}
\end{proposition}
\begin{proof}
Since there is a point $\mb u\in\mathcal R$ with only one nonzero
coordinate, Theorem~\ref{nembed} implies that $N$ embeds in $\C^m$
if only if $\gamma_i$ generates the same lattice as the whole set
$\gamma_1,\ldots,\gamma_m$ for each~$i$. Therefore,
$\gamma_1=\cdots=\gamma_m$. In this case, $D_\Gamma\cong\Z/2$ acts
by the standard antipodal involution on $S^{m-1}$, which preserves
orientation if $m$ is even and reverses orientation otherwise.
\end{proof}

Both examples of $H$-minimal Lagrangian embeddings given by
Proposition~\ref{1qemb} are well known. In fact, $S^{m-1}\times
S^1$ admits a Lagrangian embedding in~$\C^m$ for odd $m$
(see~\cite{nemi09}), but we do not know if it can be made
$H$-minimal. The Klein bottle $\mathcal K^m$ with even $m$ does
not admit Lagrangian embeddings in~$\C^m$ (see~\cite{nemi09}
and~\cite{Sh}).
\end{example}

\begin{example}[two quadrics]
In the case $m-n=2$, the topology of $\mathcal R$ and $N$ can be
described completely by analysing the action of the two commuting
involutions on the intersection of quadrics.

First, using Proposition~\ref{propcf}, we write $\mathcal R$ in
the form
\begin{equation}\label{2q}
\begin{aligned}
  \gamma_{11}u_1^2+\cdots+\gamma_{1m}u_m^2&=c_1,\\
  \gamma_{21}u_1^2+\cdots+\gamma_{2m}u_m^2&=0,
\end{aligned}
\end{equation}
where $c_1>0$ and $\gamma_{1i}>0$ for all $i$.

\begin{proposition}
There is a number $p$, \ $0<p<m$, such that $\gamma_{2i}>0$ for
$i=1,\ldots,p$ and $\gamma_{2i}<0$ for $i=p+1,\ldots,m$
in~\eqref{2q}, possibly after a reordering of the coordinates
$u_1,\ldots,u_m$. The corresponding manifold $\mathcal R=\mathcal
R(p,q)$ with $q=m-p$ is diffeomorphic to $S^{p-1}\times S^{q-1}$.
Its quotient polytope $P$ either coincides with $\Delta^{m-2}$ (if
one of the inequalities in~\eqref{ptope} is redundant) or is
combinatorially equivalent to the product
$\Delta^{p-1}\times\Delta^{q-1}$ (if there are no redundant
inequalities).
\end{proposition}
\begin{proof}
We observe that $\gamma_{2i}\ne0$ for all~$i$ in~\eqref{2q}, as
$\gamma_{2i}=0$ implies that $\mb c$ is in the cone generated by
one vector $\gamma_i$, which contradicts
Proposition~\ref{rzsmooth}~(b). By reordering the coordinates, we
can achieve that the first $p$ of the numbers $\gamma_{2i}$ are
positive and the rest are negative. Then $1<p<m$, because
otherwise~\eqref{2q} is empty. Now,~\eqref{2q} is the intersection
of the cone over the product of two ellipsoids of dimensions $p-1$
and $q-1$ (given by the second quadric) with an
$(m-1)$-dimensional ellipsoid (given by the first quadric).
Therefore, $\mathcal R(p,q)\cong S^{p-1}\times S^{p-1}$. The
statement about the polytope follows from the combinatorial fact
that a simple $n$-polytope with up to $n+2$ facets is
combinatorially equivalent to a product of simplices (see,
e.g.,~\cite[Example~I.8]{pano10}); the case of one redundant
inequality corresponds to $p=1$ or $q=1$.
\end{proof}

Now, consider the action of $D_\Gamma\cong(\Z/2)^2$ on~$\mathcal
R(p,q)$. An element $\varphi\in D_\Gamma=\frac12L^*/L^*$ acts on
$\mathcal R$ by
\[
  (u_1,\ldots,u_m)\mapsto
  (\varepsilon_1(\varphi)u_1,\ldots,\varepsilon_m(\varphi)u_m),
\]
where $\varepsilon_k(\varphi)=e^{2\pi
i\langle\gamma_k,\varphi\rangle}=\pm1$ for $1\le k\le m$.

\begin{lemma}\label{free1}
Suppose that $D_\varGamma$ acts on $\mathcal R(p,q)$ freely and
$\varepsilon_i(\varphi)=1$ for some $\varphi\in D_\varGamma$ and
$1\le i\le p$. Then $\varepsilon_j(\varphi)=-1$ for $p+1\le j\le
m$.
\end{lemma}
\begin{proof}
Assume the opposite, that is, that $\varepsilon_i(\varphi)=1$ for
some $1\le i\le p$ and $\varepsilon_j(\varphi)=1$ for some $p+1\le
j\le m$. Then $\gamma_{2i}>0$ and $\gamma_{2j}<0$ in~\eqref{2q},
so we can choose $\mb u\in\mathcal R(p,q)$ whose only nonzero
coordinates are $u_i$ and~$u_j$. The element $\varphi\in D_\Gamma$
fixes this $\mb u$, leading to a contradiction.
\end{proof}

\begin{lemma}\label{free2}
Suppose that $D_\Gamma$ acts on $\mathcal R(p,q)$ freely. Then
there exist two generating involutions $\varphi_1,\varphi_2\in
D_\Gamma\cong\Z_2\times\Z_2$ whose action on $\mathcal R(p,q)$ is
described by either~{\rm(1)} or~{\rm(2)} below, possibly after a
reordering of coordinates:
\begin{itemize}
\item[(1)]
$\begin{aligned}
  \varphi_1\colon(u_1,\ldots,u_m)&\mapsto
  (u_1,\ldots,u_k,-u_{k+1},\ldots,-u_p,-u_{p+1},\ldots,-u_m),\\[-2pt]
  \varphi_2\colon(u_1,\ldots,u_m)&\mapsto
  (-u_1,\ldots,-u_k,u_{k+1},\ldots,u_p,-u_{p+1},\ldots,-u_m);
\end{aligned}$\\
\item[(2)]
$\begin{aligned}
  \varphi_1\colon(u_1,\ldots,u_m)&\mapsto
  (-u_1,\ldots,-u_p,u_{p+1},\ldots,u_{p+l},-u_{p+l+1},\ldots,-u_m),\\[-2pt]
  \varphi_2\colon(u_1,\ldots,u_m)&\mapsto
  (-u_1,\ldots,-u_p,-u_{p+1},\ldots,-u_{p+l},u_{p+l+1},\ldots,u_m);
\end{aligned}$
\end{itemize}
here $0\le k\le p$ and $0\le l\le q$.
\end{lemma}
\begin{proof}
By Lemma~\ref{free1}, for each of the three nonzero elements
$\varphi\in D_\Gamma$, we have either $\varepsilon_i(\varphi)=-1$
for $1\le i\le p$ or $\varepsilon_i(\varphi)=-1$ for $p+1\le i\le
m$. Therefore, we can choose two different nonzero elements
$\varphi_1,\varphi_2\in D_\Gamma$ such that either
$\varepsilon_i(\varphi_j)=-1$ for $j=1,2$ and $p+1\le i\le m$, or
$\varepsilon_i(\varphi_j)=-1$ for $j=1,2$ and $1\le i\le p$. This
corresponds to the cases (1) and (2) above, respectively. In the
former case, after reordering the coordinates, we may assume that
$\varphi_1$ acts as in~(1). Then $\varphi_2$ also acts as in~(1),
since otherwise the sum $\varphi_1+\varphi_2$ cannot act freely by
Lemma~\ref{free1}. The second case is treated similarly.
\end{proof}

Each of the actions of $D_\Gamma$ described in Lemma~\ref{free2}
can be realised by a particular intersection of
quadrics~\eqref{2q}. For example,
\begin{equation}\label{2qex}
\begin{aligned}
  2u_1^2+\cdots+2u_k^2+u_{k+1}^2+\cdots+u_p^2+
  u_{p+1}^2+\cdots+u_m^2&=3,\\
  u_1^2+\cdots+u_k^2+2u_{k+1}^2+\cdots+2u_p^2-
  u_{p+1}^2-\cdots-u_m^2&=0
\end{aligned}
\end{equation}
gives the first action of Lemma~\ref{free2}; the second action is
realised similarly. Note that the lattice $L$ corresponding
to~\eqref{2qex} is a sublattice of index 3 in~$\Z^2$. We can
rewrite~\eqref{2qex} as
\begin{equation}\label{2qex1}
\begin{aligned}
  u_1^2+\cdots+u_k^2&+u_{k+1}^2+\cdots+u_p^2&&=1,\\
  u_1^2+\cdots+u_k^2&&+u_{p+1}^2+\cdots+u_m^2&=2,
\end{aligned}
\end{equation}
in which case $L=\Z^2$. The action of the two involutions
$\psi_1,\psi_2\in D_\Gamma=\frac12\Z^2/\Z^2$ corresponding to the
standard basis vectors of $\frac12\Z^2$ is given by
\begin{equation}\label{2inv}
\begin{aligned}
  \psi_1\colon(u_1,\ldots,u_m)&\mapsto
  (-u_1,\ldots,-u_k,-u_{k+1},\ldots,-u_p,u_{p+1},\ldots,u_m),\\
  \psi_2\colon(u_1,\ldots,u_m)&\mapsto
  (-u_1,\ldots,-u_k,u_{k+1},\ldots,u_p,-u_{p+1},\ldots,-u_m).
\end{aligned}
\end{equation}

We denote the manifold $N_\Gamma$ corresponding to~\eqref{2qex1}
by $N_k(p,q)$. We have
\begin{equation}\label{nkpq}
  N_k(p,q)\cong\mathcal (S^{p-1}\times
  S^{q-1})\times_{\Z/2\times\Z/2}(S^1\times S^1),
\end{equation}
and the action of the two involutions on $S^{p-1}\times S^{q-1}$
is given by the $\psi_1$ and $\psi_2$ defined by~\eqref{2inv}.
Note that $\psi_1$ acts trivially on $S^{q-1}$ and acts
antipodally on $S^{p-1}$. Therefore,
\[
  N_k(p,q)\cong N(p)\times_{\Z/2}(S^{q-1}\times S^1),
\]
where $N(p)$ is the manifold from Proposition~\ref{1qemb}. If
$k=0$, then the second involution $\psi_2$ acts trivially on
$N(p)$, and $N_0(p,q)$ coincides with the product $N(p)\times
N(q)$ of the two manifolds from Example~\ref{1quad}. In general,
the projection $N_k(p,q)\to S^{q-1}\times_{\Z/2}S^1=N(q)$
describes $N_k(p,q)$ as the total space of a fibration over $N(q)$
with fibre~$N(p)$.

We summarise the above facts and observations in the following
topological classification result for compact $H$-minimal
Lagrangian submanifolds $N\subset\C^m$ obtained from intersections
of two quadrics.

\begin{theorem}\label{2qemb}
Assume that $m-n=2$ (two quadrics) and $N_\Gamma\to\C^m$ is the
embedding of the corresponding $H$-minimal Lagrangian submanifold.
Then $N_\Gamma$ is diffeomorphic to some $N_k(p,q)$ given
by~\eqref{nkpq}, where $p+q=m$, $0<p<m$ and $0\le k\le p$.
Moreover, any such triple $(k,p,q)$ can be realised by~$N_\Gamma$.
\end{theorem}
\end{example}


In the case of up to two quadrics considered above, the topology
of $\mathcal R_\Gamma$ is relatively simple, and in order to
analyse the topology of $N_\Gamma$, one only needs to describe the
action of involutions on~$\mathcal R_\Gamma$. When the number of
quadrics is more than two, the topology of $\mathcal R_\Gamma$
becomes an issue as well.

\begin{example}[three quadrics]\label{3quad}
In the case $m-n=3$, the topology of compact manifolds $\mathcal
R$ and $\mathcal Z$ was fully described
in~\cite[Theorem~2]{lope89}. Each of these manifolds is
diffeomorphic to a product of three spheres or to a connected sum
of products of spheres with two spheres in each product.

Note that, for $m-n=3$, the manifolds $\mathcal R_P$ (or~$\mathcal
Z_P$) can be distinguished topologically by looking at the planar
\emph{Gale diagrams} of the corresponding simple polytopes~$P$
(see~\cite{bo-me06} or~\cite{pano10} for details). This chimes
with the well-known classification of $n$-dimensional simple
polytopes with $n+3$ facets.

The smallest polytope with $m-n=3$ is a pentagon. It has many
Delzant realisations, for instance,
\[
  P=\bigl\{(x_1,x_2)\in\R^2\colon x_1\ge0,\;x_2\ge0,\;-x_1+2\ge0,\;-x_2+2\ge0,\;
  -x_1-x_2+3\ge0\bigr\}.
\]
In this case, $\mathcal R_P$ is an oriented surface of genus~5 (a
simple combinatorial proof of this can be found
in~\cite[Example~6.40]{bu-pa02}), and $\mathcal Z_P$ is
diffeomorphic to a connected sum of five copies of $S^3\times
S^4$.

We therefore obtain an $H$-minimal Lagrangian submanifold
$N_P\subset\C^5$, which is the total space of a bundle over $T^3$
with fibre a surface of genus~5.
\end{example}

More generally, the manifolds $\mathcal R_P$ corresponding to
polygons are described as follows.

\begin{proposition}[{\cite[Example~6.40]{bu-pa02}}]\label{n=2}
Assume that $n=2$ in~\eqref{rquad} and the 2-dimensional polytope
$P$ corresponding to~$\mathcal R$ is an $m$-gon, i.e., there are
no redundant inequalities. Then $\mathcal R$ is an orientable
surface $S_g$ of genus $g=1+2^{m-3}(m-4)$.
\end{proposition}

The moment-angle manifold $\mathcal Z$ corresponding to a polygon
is a connected sum of sphere products~\cite[Theorem~6.3]{bo-me06}.

If $n=2$ and there are $k$ redundant inequalities
in~\eqref{ptope}, then $P$ is an $(m-k)$-gon. In this case
$\mathcal R\cong\mathcal R'\times(S^0)^k$, where $\mathcal R'$
corresponds to an $(m-k)$-gon without redundant inequalities. That
is, $\mathcal R$ is a disjoint union of $2^k$ surfaces of genus
$1+2^{m-k-3}(m-k-4)$. Similarly, $\mathcal Z\cong\mathcal
Z'\times(S^1)^k$.

The $H$-minimal Lagrangian submanifold $N\subset\C^m$
corresponding to $\mathcal R$ from Proposition~\ref{n=2} is the
total space of a bundle over $T^{m-2}$ with fibre~$S_g$. This is
an aspherical manifold (for $m\ge4$) whose fundamental group is
included in the short exact sequence
\[
\begin{CD}
  1@>>>\pi_1(S_g)@>>>\pi_1(N)@>>>\Z^{m-2}@>>>1.
\end{CD}
\]

For $n>2$ and $m-n>3$ the topology of $\mathcal R$ and $\mathcal
Z$ is even more complicated, see~\cite{bu-pa02},~\cite{gi-lo09}
and~\cite[\S{}III.2]{pano10} for more results in this direction.

\end{document}